\documentclass[11pt,a4paper]{amsart}
\usepackage{latexsym}
\usepackage{graphicx}
\usepackage{subfig}
\usepackage{caption}
\usepackage{float}
\usepackage{enumerate}
\usepackage[top=3.2cm,bottom=3.8cm,left=3cm,right=2cm]{geometry}
\usepackage{mathrsfs}
\usepackage{amssymb}
\usepackage{amsbsy}
\usepackage[colorlinks,linkcolor=blue,citecolor=blue,pagebackref]{hyperref}

\newtheorem{theorem}{Theorem}[section]
\newtheorem{lemma}[theorem]{Lemma}

\theoremstyle{definition}

\theoremstyle{remark}
\newtheorem{remark}[theorem]{Remark}

\numberwithin{equation}{section}

\newcommand{\R}{\mathbb{R}}
\newcommand{\Sph}{\mathbb{S}}
\newcommand{\eps}{\varepsilon}

\newcommand{\Cov}{\operatorname{Cov}}
\newcommand{\Var}{\operatorname{Var}}

\newcommand{\Hess}{\nabla^2}
\newcommand{\HS}{\mathrm{HS}}
\newcommand{\Id}{\mathrm{Id}}

\begin{document}
	
	\begin{center}
		{\large\bf A uniform bound in the dimensional Brunn--Minkowski inequality for even log-concave measures}
	\end{center}
	
	\vskip 15pt
	\begin{center}
		{\small\bf Kai-Wen Yang}\\~~ \\
		\small{School of Mathematical Sciences, Key Laboratory of Intelligent Computing and Applications (Ministry of Education), Tongji University, Shanghai, 200092, China}
	\end{center}
	
	\vskip 5pt
	\begin{NoHyper}
		\footnotetext{E-mail addresses:  yangkaiwen@tongji.edu.cn.}
	\end{NoHyper}
	
	\begin{center}
		\begin{minipage}{14cm}
			{{\bf Abstract:}
				For every $n\ge 2$, we prove that there exists an exponent $p_n$ such that, for every even log-concave probability measure $\mu$ on $\mathbb R^n$, all nonempty symmetric convex sets $K,L\subseteq\mathbb R^n$, and all $\lambda\in[0,1]$,
				\[
				\mu(\lambda K+(1-\lambda)L)^{p_n}
				\ge
				\lambda\mu(K)^{p_n}+(1-\lambda)\mu(L)^{p_n},
				\]
				where 
				\[
				p_n\ge \frac{c}{n^2\ln n}
				\]
				for some absolute constant $c>0$.}
			
			
		\end{minipage}
	\end{center}
	
	\vskip 20pt
	
	\section{\bf Introduction}
	\vskip 5pt
	
	The setting for this paper is the $n$-dimensional Euclidean space, $\mathbb{R}^n$. Let $V_n$ be the $n$-dimensional Lebesgue measure. The classical Brunn–Minkowski inequality asserts that for all Borel sets $K, L \subseteq \mathbb{R}^n$ and all $\lambda \in [0,1]$, 
	$$
	V_n(\lambda K+(1-\lambda) L)^\frac{1}{n} \geq \lambda V_n(K)^\frac{1}{n}+ (1-\lambda) V_n(L)^\frac{1}{n}.
	$$
	Here, $\lambda K+(1-\lambda) L=\left\{\lambda x+(1-\lambda) y: x \in K, y \in L\right\}$ is the Minkowski combination of $K$ and $L$. The Brunn--Minkowski inequality says that the $n$-dimensional Lebesgue measure is $\frac{1}{n}$-concave with respect to Minkowski combinations for Borel sets.
	
	The Brunn--Minkowski inequality is one of the cornerstones of convex geometry. See, e.g., books by Gardner \cite{Gardner-2006} and Schneider \cite{Schneider-book}.  
	It is also closely related to
	many other fundamental inequalities, such as the isoperimetric inequality, the Prékopa-Leindler inequality, the Sobolev inequality and the Brascamp-Lieb inequality. See, e.g., Barthe \cite{Barthe} and Bobkov and Ledoux \cite{Bobkov1,Bobkov2}. We refer to the excellent survey by Gardner \cite{Gardner-2002} for a more comprehensive account.
	
	Extending this dimensional concavity property beyond Lebesgue measure is a central topic in modern convex geometry. Recall that a Borel measure $\mu$ on $\R^n$ is log-concave if it has density $e^{-V}$, where $V$ is convex. The dimensional Brunn–Minkowski	conjecture asks whether for every even log-concave probability measure $\mu$ on $\R^n$, 
	\begin{equation}\label{eq:logconcave}
		\mu(\lambda K+(1-\lambda) L)^\frac{1}{n} \geq \lambda \mu(K)^\frac{1}{n}+ (1-\lambda) \mu(L)^\frac{1}{n}
	\end{equation}
	for all nonempty symmetric convex sets $K, L \subseteq \mathbb{R}^n$ and all $\lambda \in [0,1]$?
	
	This conjecture was first posed by Gardner and Zvavitch \cite{Gardner-Zvavitch-2010} for Gaussian measures and later formulated in this generality by Colesanti, Livshyts and Marsiglietti \cite{CLM17}. It is confirmed by Eskenazis and Moschidis \cite{Eskenazis-Moschidis-2021} for Gaussian measures, and  by Cordero-Erausquin and Rotem \cite{Cordero-Rotem-2023} for all rotationally invariant log-concave measures. Livshyts, Marsiglietti, Nayar and Zvavitch \cite[Proposition 1]{LMNZ17} showed that this conjecture follows from the logarithmic Brunn–Minkowski conjecture posed by Böröczky, Lutwak, Yang and Zhang \cite{BLYZ12}, so it is valid in dimension $n=2$ by Böröczky et al. \cite{BLYZ12}, and for unconditional (that is, symmetric with
	respect to every coordinate hyperplane) convex bodies and unconditional log-concave probability measures by Saroglou \cite{Saroglou2015,Saroglou2016}. See \cite{Aishwarya-Li-2025,Aishwarya-Rotem-2026,Cordero-Erausquin-Eskenazis-preprint,MMRR-preprint} for the entropic or functional forms of this conjecture.
	
	The dimensional Brunn--Minkowski conjecture remains open for general even log-concave measures in dimensions \(n\geq 3\). In 2023, Livshyts \cite{Livshyts-2023} proved an inequality of the form \eqref{eq:logconcave} with a universal lower bound \(\frac{1}{n^{4+o(1)}}\) on the exponent. Very recently, Eskenazis, Giannopoulos and Tziotziou \cite{AAN} showed that one may take this exponent as \(\frac{c}{n^3\ln n}\), where \(c>0\) is an absolute constant. 
	In this paper, we improve the lower bound on the exponent by a factor of \(n\). Our main result is the following.
	
	\begin{theorem}\label{main}
		Let $\mu$ be an even log-concave probability measure on $\R^n$, $n\ge2$. Then, for all nonempty symmetric convex sets $K,L\subseteq\R^n$ and all $\lambda\in[0,1]$,
		\begin{equation}\label{eq:mainBM}
			\mu(\lambda K+(1-\lambda)L)^{p_n}
			\ge
			\lambda\mu(K)^{p_n}+(1-\lambda)\mu(L)^{p_n},
		\end{equation}
		where
		$$
		p_n\ge \frac{c}{n^2\ln n}
		$$
		for some absolute constant $c>0$.
	\end{theorem}
	
	Our proof follows the analytic route used in \cite{AAN,Livshyts-2023}, but introduces a new energy estimate. As pointed out by Eskenazis, Giannopoulos and Tziotziou \cite[Remark~3.8]{AAN}, a set of large measure on which the gradient of the potential is controlled need not be convex, whereas convexity is required in their argument to control the Poincaré constant of the restricted measure. Our key observation is that one may instead apply the Poincaré inequality on the entire support before integrating the equation over the low-gradient set. This leads to a new energy estimate 	(Lemma~\ref{lem:energy}), which bypasses the convexity obstruction and improves the exponent obtained in \cite{AAN} by a factor of \(n\).
	
	This paper is organized as follows. In Section~\ref{sec:prelim}, we recall the required facts about convex bodies, log-concave measures and the variation method. In Section~\ref{sec:proof}, we establish the new energy estimate and prove Theorem~\ref{main}.
	
	\vskip 20pt
	
	\section{\bf Preliminaries}\label{sec:prelim}
	\vskip 5pt
	
	All constants denoted by $c,C,c_1,C_1,\ldots$ are positive absolute constants whose value may change from line to line. We write $|x|$ for the Euclidean norm, $\langle x,y\rangle$ for the Euclidean inner product, $B_2^n$ for the Euclidean unit ball and $\Sph^{n-1}=\partial B_2^n$. All convex sets are assumed to be nonempty below.
	
	\subsection{Convex bodies}
	\
	
	A convex body is a compact convex subset of $\R^n$ with nonempty interior. A convex body $K$ is symmetric if $K=-K$. We write $\mathcal K_s^n$ for the class of symmetric convex bodies. Good references on convex bodies are the books by Gardner \cite{Gardner-2006} and Schneider \cite{Schneider-book}.
	
	The support function of a convex body $K$ is
	\[
	h_K(u)=\max_{x\in K}\langle x,u\rangle,\qquad u\in\Sph^{n-1}.
	\]
	The Hausdorff distance between convex bodies $K$ and $L$ in $\R^n$ is 
	\[
	d_H(K,L)=  \left\|h_K-h_L\right\|_{C^0\left(\mathbb{S}^{n-1}\right)}.
	\]
	If $K$ and $L$ are convex bodies in $\mathbb{R}^n$, then for $\alpha, \beta\ge 0$,
	$$
	h_{\alpha K+\beta L}=\alpha	h_{K}+\beta	h_{L}.
	$$
	
	A convex body $K$ in $\R^n$ is in the class of $\mathcal C^2_+$ if its boundary is a $C^2$ hypersurface with everywhere positive Gauss curvature. 
	We denote by $\mathcal K^n_{+,s}$ the class of $\mathcal C^2_+$ symmetric convex bodies  in $\mathbb{R}^n$.
	Let \( C_e^2(\mathbb{S}^{n-1}) \) be the set of even \( C^2 \) functions on \( \mathbb{S}^{n-1} \). For any \( \psi \in C_e^2(\mathbb{S}^{n-1}) \) and \( K \in \mathcal K^n_{+,s} \), the convex body \( K_s \) with support function
	\(
	h_{K_s} = h_K + s\psi
	\)
	is also in \( \mathcal K^n_{+,s} \) when \( s \) is sufficiently small. Thus,
	\[
	C_e^2(\mathbb{S}^{n-1})=\{h_L-h_K: K,L\in \mathcal{K}_{+,s}^n\}.
	\]
	
	Every convex body can be approximated in the Hausdorff metric by bodies in $\mathcal C^2_+$, and the approximation may be chosen symmetric when the original body is symmetric. 
	
	\subsection{Log-concave measures}
	\
	
	A Borel measure $\mu$  on $\mathbb{R}^n$ is log-concave if it has density 
	\(
	d\mu=e^{-V}dx,
	\)
	where $V:\R^n\to(-\infty,+\infty]$ is convex, and it is a probability measure if $\mu(\R^n)=1$. The value $+\infty$ is allowed, so restrictions of log-concave measures to convex sets are also log-concave. The measure is even if $V(x)=V(-x)$, equivalently $\mu(A)=\mu(-A)$ for every Borel set $A$. 
	
	Let \(\mu\) and \(\nu\) be Borel probability measures on \(\mathbb R^n\). Their convolution \(\mu*\nu\) is the Borel probability measure defined by
	\[
	(\mu*\nu)(A)
	=
	\int_{\mathbb R^n}\nu(A-x)\,d\mu(x)
	=
	\iint_{\mathbb R^n\times\mathbb R^n}
	\mathbf 1_A(x+y)\,d\mu(x)d\nu(y)
	\]
	for every Borel set \(A\subset\mathbb R^n\), where
	\(
	A-x=\{z\in\mathbb R^n:x+z\in A\}.
	\)
	
	For \(\eps>0\), let \(\gamma_\eps\) denote the centered Gaussian probability measure on \(\mathbb R^n\) with covariance matrix \(\eps^2\Id\), i.e., 
	\(
	d\gamma_\eps(x)
	=
	\frac{1}{(2\pi\eps^2)^{\frac{n}{2}}}
	\exp(-\frac{|x|^2}{2\eps^2})dx.
	\) If \(\mu\) is a probability measure on \(\mathbb R^n\), then the convolution
	\(
	\mu_\eps:=\mu*\gamma_\eps
	\) is also a probability measure and has a \(C^\infty\) density. 	Moreover, if \(\mu\) is even, then \(\mu_\eps\) is also even for $\eps>0$.
	
	For \(a\in\mathbb R^n\), denote the Dirac mass at \(a\) by \(\delta_a\).
	As \(\eps\to0\), \(\gamma_\eps\)  converges weakly to \(\delta_0\).
	Consequently, if \(\mu\) is a probability measure on \(\mathbb R^n\), then \(	\mu*\gamma_\eps\)  converges weakly to \(\mu\).

	We shall use the theorem of Pr\'ekopa  \cite{Prekopa} and convolution method to regularize measures.
	
	\begin{theorem}[\cite{Prekopa}]\label{thm:prekopa}
		Let $F:\R^n\times\R^m\to[0,\infty)$ be log-concave. Then the marginal
		\[
		G(x)=\int_{\R^m}F(x,y)\,dy
		\]
		is log-concave on its support, whenever the integral is finite. In particular, the convolution of two log-concave densities is log-concave.
	\end{theorem}
	
	\vskip 3pt

	A probability measure $\mu$ on $\R^n$ is centered if 
	\[
	\int_{\R^n}x\,d\mu(x)=0,
	\]
	and is isotropic if in addition
	\[
	\int_{\R^n} x_i x_j\,d\mu(x)=\delta_{ij}.
	\]
	Every centered log-concave probability measure can be put into isotropic position by an invertible linear map. More precisely, if $\mu$ is centered, set $T=\Cov(\mu)^{-\frac{1}{2}}$, then $T_\#\mu$ is isotropic, where \(\Cov(\mu)\) is the symmetric positive definite matrix
	\[
	\Cov(\mu)
	=
	\int_{\mathbb R^n}x\otimes x\,d\mu(x),
	\]
	and $T_\#\mu$ is the push-forward of $\mu$ under $T$ given by 
	\[
	(T_\#\mu)(A)=\mu(T^{-1}(A)),
	\]
	for every Borel set \(A\subset\mathbb R^n\).
	Even log-concave probability measures are centered.
	
	The Poincar\'e constant $\vartheta_\mu$ of a probability measure $\mu$ on $\R^n$ is defined as the smallest $\vartheta>0$ such that
	\begin{equation}\label{eq:poincare}
		\Var_\mu(f)\le \vartheta^2\int_{\R^n} |\nabla f|^2\,d\mu
	\end{equation}
	for all locally Lipschitz functions $f$ on $\R^n$. The celebrated Kannan--Lov\'asz--Simonovits conjecture \cite{KLS} asks whether there exists an absolute constant $C>0$ such that 
	\[
	\vartheta_\mu\le C
	\]
	for all isotropic log-concave probability measures. Indeed, this conjecture was originally formulated in terms of a Cheeger inequality rather than a Poincar\'e inequality, but the two formulations turn out to be equivalent. Please refer to Buser \cite{Buser1982ASENP}, Ledoux \cite{Ledoux1994PAMS} and Milman \cite{milman2009invent}.
	
	The best currently known estimate for $\vartheta_\mu$ on isotropic log-concave probability measures is due to Klartag \cite{Klartag-2023}. For more progress, please refer to the survey of Klartag and Lehec \cite{KL2025BAMS}.
	
	\begin{theorem}[\cite{Klartag-2023}]\label{thm:KLS}
		There exists an absolute constant $C>0$ such that every isotropic log-concave probability measure $\mu$ on $\R^n$, $n\ge2$, satisfies
		\[
		\vartheta_\mu\le C\sqrt{\ln n}.
		\]
	\end{theorem}
	
	We shall also use the following $L^1$ estimate on the gradient of the logarithmic potential $V$, due to Eldan and Klartag \cite[Lemma~11]{Eldan-Klartag-2008}; see also \cite[Theorem 1.3]{AAN}.
	
	\begin{lemma}[\cite{Eldan-Klartag-2008}]\label{thm:EK}
		There exists an absolute constant $C>0$ such that every isotropic log-concave probability measure $\mu$ with $d\mu=e^{-V}dx$ on $\R^n$ satisfies
		\begin{equation}\label{eq:EK}
			\int_{\R^n} |\nabla V|\,d\mu\le Cn.
		\end{equation}
	\end{lemma}
	
	The order $n$ in \eqref{eq:EK} is optimal; see \cite[Proposition 5.7]{AAN}.
	\begin{lemma}[\cite{AAN}]\label{optimal}
		There exists an even isotropic log-concave probability measure $\mu$ with $d\mu=e^{-V}dx$ on $\mathbb{R}^n$ such that
		$$
		\int_{\mathbb{R}^n}|\nabla V| d \mu \geqslant c n,
		$$
		where $c>0$ is an absolute constant.
	\end{lemma}
	
	\subsection{The variation method}
	\
	
	For this subsection assume that $d\mu=e^{-V}dx$ is an even log-concave measure with positive $C^2$ density. If $K\in\mathcal K^n_{+,s}$, define $p(\mu,K)$ as the largest number $p\ge0$ as in \cite{AAN} such that, for every $L\in\mathcal K^n_{+,s}$,
	\begin{equation}\label{eq:localpower}
		\frac{d^2}{d\lambda^2}\mu(\lambda K+(1-\lambda)L)^p\Big|_{\lambda=1}\le0.
	\end{equation}
	
	The following affine invariance is needed; see \cite[Lemma 3.1.]{AAN}.
	
	\begin{lemma}\label{lem:affine}
		Let $T:\R^n\to\R^n$ be invertible. Then
		\[
		p(\mu,K)=p(T_\#\mu,TK)
		\]
		for every even log-concave probability measure $\mu$ with positive $C^2$ density and every $K\in\mathcal K^n_{+,s}$.
	\end{lemma}
	
	The local-to-global principle, see, e.g., in Kolesnikov--Livshyts \cite[Lemma 3.1]{Kolesnikov-Livshyts-2021} or Kolesnikov--Milman \cite[Section 3]{Kolesnikov-Milman-2022}, gives the following relationship.
	
	\begin{lemma}\label{lem:lTog}
		Let $\mu$ be an even log-concave probability measure on $\mathbb{R}^n$ with positive $C^2$ density and $p>0$. Then
		\begin{equation}\label{eq:localglobal}
			\inf_{K\in\mathcal K^n_{+,s}}p(\mu,K)\ge p
		\end{equation}
		if and only if
		\[
		\mu(\lambda K+(1-\lambda)L)^p
		\ge
		\lambda\mu(K)^p+(1-\lambda)\mu(L)^p
		\]
		for all symmetric convex bodies $K,L$ in $\mathbb{R}^n$ and all $\lambda\in[0,1]$. 
	\end{lemma}

	Let
	\[
	L_Vu=\Delta u-\langle\nabla V,\nabla u\rangle.
	\]
	For $K \in \mathcal{K}_{+, s}^n$, let $\nu_K: \partial K \rightarrow \mathbb{S}^{n-1}$ be the Gauss map of $K$. Let \( \mathcal{H}_{\partial K}^{n-1} \) be the restriction of \( \mathcal{H}^{n-1} \) to the boundary \( \partial K \) and \(d\mu_{\partial K} = e^{-V} d\mathcal{H}_{\partial K}^{n-1}\). For \( K \in \mathcal K^n_{+,s} \) and \(f\in C^1(\partial K)\), it was shown by Kolesnikov and Milman \cite{Kolesnikov-Milman-2018} that  the Neumann boundary problem
	\begin{equation}\label{eq:NP}
		\begin{cases}L_V u=1, & \text { in } K, \\ \langle\nabla u, \nu_K\rangle=f, & \text { on } \partial K\end{cases}
	\end{equation}
	has a unique solution $u\in C^2(\operatorname{int} K)\cap C^1(K)$ up to an
	additive constant if
	\begin{equation}\label{eq:compatibility}
		\int_{\partial K} f d \mu_{\partial K}=\mu(K).
	\end{equation}
	Moreover, if $K$ is symmetric and $V$ and $f$ are even, the solution $u$ is even.
	
	For $K,L\in \mathcal{K}^n_{+,s}$, write $\psi=h_L-h_K$ and $\varphi=\psi\circ \nu_K$.  Then $$h_{\lambda K+(1-\lambda)L}=h_K+(1-\lambda)\psi.$$
	Write $\Pi$ for the second fundamental form of \(\partial K\), and \(H_V = \operatorname{tr}(\Pi) - \langle \nabla V, \nu_K \rangle\) for the weighted mean curvature of \(\partial K\). The following variational formulas were obtained by Kolesnikov and Milman \cite{Kolesnikov-Milman-2018} in their proof of Theorem 6.6.
	\begin{equation}\label{eq:VF}
		\begin{aligned}
			\frac{d}{d\lambda}\mu(\lambda K+(1-\lambda)L)\Big|_{\lambda=1}& = -\int_{\partial K} \varphi  d\mu_{\partial K},\\
			\frac{d^2}{d\lambda^2}\mu(\lambda K+(1-\lambda)L)\Big|_{\lambda=1}& = \int_{\partial K} \big( H_V \varphi^2 - \langle \Pi^{-1} \nabla_{\partial K} \varphi, \nabla_{\partial K} \varphi \rangle \big)  d\mu_{\partial K}.
		\end{aligned}
	\end{equation}
	These formulas generalize the formulas obtained by Colesanti \cite{Colesanti-2008} for Lebesgue measure.
	
	Let $u_{\nu_K}=\langle \nabla u , \nu_K \rangle$. For $u\in C^2(\operatorname{int} K)\cap C^1(K)$ with $u_{\nu_K}\in C^1(\partial K)$, the following weighted Reilly formula was proved by Kolesnikov and Milman \cite{milman2},
	\begin{equation}\label{eq:relliy}
		\begin{aligned}
			\int_K (L_Vu)^2 d \mu&=\int_K\big(\left\|\nabla^2 u\right\|_{\HS}^2+\left\langle\nabla^2 V \nabla u, \nabla u\right\rangle\big) d \mu \\
			& +\int_{\partial K}\left(H_V u_{\nu_K}^2-2\left\langle\nabla_{\partial K} u, \nabla_{\partial K} u_{\nu_K}\right\rangle+\left\langle\Pi \nabla_{\partial K} u, \nabla_{\partial K} u\right\rangle\right) d \mu_{\partial K}.
		\end{aligned}
	\end{equation}

	The following result shown by Kolesnikov and Livshyts \cite[Proof of Lemma 2.3]{Kolesnikov-Livshyts-2021} will be used to estimate $p(\mu,K)$. We include the proof here for completeness.
	
	\begin{theorem}[{\cite{Kolesnikov-Livshyts-2021}}]\label{thm:KM}
		Let $\mu$ be an even log-concave measure on $\mathbb{R}^n$ with positive $C^2$ density $e^{-V}$, $K\in\mathcal K^n_{+,s}$ and $p>0$. Assume that every even $u\in C^2(\operatorname{int} K)$ with $L_Vu=1$ in $K$ satisfies 
		\begin{equation}\label{eq:KMcriterion}
			\frac{1}{\mu(K)}
			\int_K
			\left(\|\Hess u\|_{\HS}^2+
			\langle\Hess V\nabla u,\nabla u\rangle\right)d\mu
			\ge p.
		\end{equation}
		Then $p(\mu,K)\ge p$.
	\end{theorem}
	
	\begin{proof}
		Fix $L \in \mathcal{K}_{+, s}^n$ and set $\varphi=(h_L-h_K) \circ \nu_K$. By \eqref{eq:VF} and the chain rule applied to $t \mapsto t^p$, it is enough to prove
		\[
		\begin{aligned}
			0\ge \mu(K)\int_{\partial K}\left(H_V\varphi^2-\langle\Pi^{-1}\nabla_{\partial K}\varphi,\nabla_{\partial K}\varphi\rangle\right)d\mu_{\partial K} +(p-1)\Big(\int_{\partial K}\varphi\,d\mu_{\partial K}\Big)^2.
		\end{aligned}
		\]
		Assume first that \(\int_{\partial K}\varphi\,d\mu_{\partial K}\ne0\). Since $K$ and $L$ are symmetric convex bodies, $\varphi=(h_L-h_K) \circ \nu_K$ is even.   Let $u$ be the even solution of the Neumann problem \eqref{eq:NP} with boundary condition
		\[
		u_{\nu_K}=
		\frac{\mu(K)}{\int_{\partial K}\varphi\,d\mu_{\partial K}}\varphi .
		\]
		The compatibility condition \eqref{eq:compatibility} is satisfied.

		Using \eqref{eq:relliy}, \eqref{eq:KMcriterion}, and the positivity of $\Pi$, we have
		\[
		\begin{aligned}
			\mu(K)=&\int_K (L_Vu)^2 d \mu=\int_K\big(\left\|\nabla^2 u\right\|_{\HS}^2+\left\langle\nabla^2 V \nabla u, \nabla u\right\rangle\big) d \mu\\
			&+\int_{\partial K}\left(H_V u_{\nu_K}^2-2\left\langle\nabla_{\partial K} u, \nabla_{\partial K} u_{\nu_K}\right\rangle+\left\langle\Pi \nabla_{\partial K} u, \nabla_{\partial K} u\right\rangle\right) d \mu_{\partial K}\\
			\ge& p\mu(K)+ 
			\int_{\partial K}\left(H_Vu_{\nu_K}^2-\langle\Pi^{-1}\nabla_{\partial K}u_{\nu_K},\nabla_{\partial K}u_{\nu_K}\rangle\right)d\mu_{\partial K}\\
			=&p\mu(K)+
			\Big(\frac{\mu(K)}{\int_{\partial K}\varphi\,d\mu_{\partial K}}\Big)^2\cdot 
			\int_{\partial K}\left(H_V\varphi^2-\langle\Pi^{-1}\nabla_{\partial K}\varphi,\nabla_{\partial K}\varphi\rangle\right)d\mu_{\partial K}.
		\end{aligned}
		\]
		This is exactly the desired inequality.

		When \(\int_{\partial K}\varphi\,d\mu_{\partial K}=0\), let
		\[
		\varphi_\eps=\varphi+\eps h_K\circ\nu_K,\quad \eps>0.
		\]
		Then \[\int_{\partial K}\varphi_\eps\,d\mu_{\partial K}=\eps \int_{\partial K}h_K\circ\nu_K\,d\mu_{\partial K}> 0.\] By the previous case, we have
		\[
		\begin{aligned}
			0\ge \mu(K)\int_{\partial K}\left(H_V\varphi_\eps^2-\langle\Pi^{-1}\nabla_{\partial K}\varphi_\eps,\nabla_{\partial K}\varphi_\eps\rangle\right)d\mu_{\partial K} +(p-1)\Big(\int_{\partial K}\varphi_\eps\,d\mu_{\partial K}\Big)^2.
		\end{aligned}
		\]
		Letting $\eps \downarrow 0$, this gives the desired inequality.
	\end{proof}

	\vskip 20pt
	
	\section{\bf Proof of the main result}\label{sec:proof}
	\vskip 5pt
	
	We first give a new estimate for $\int\|\Hess u\|_{\HS}^2\,d\mu$, which allows us to obtain the new uniform bound in Theorem \ref{main}. 
	
	\begin{lemma}\label{lem:energy}
		Let $d\mu=e^{-V}dx$ be an even isotropic log-concave probability measure on $\R^n$, $n\ge2$, whose support is a $K\in\mathcal K^n_{+,s}$, and let  $\vartheta_\mu$ be the Poincar\'e constant of $\mu$. If $u\in C^2(\operatorname{int} K)$ is  even and satisfies
		\[
		L_Vu:=\Delta u-\langle\nabla V,\nabla u\rangle=1,\quad\text{in } K,
		\]
		then
		\begin{equation}\label{eq:energy-theta}
			\int_K\|\Hess u\|_{\HS}^2\,d\mu
			\ge
			\frac{c}{n+n^2\vartheta_\mu^2}.
		\end{equation}
		Consequently,
		\begin{equation}\label{eq:energy-log}
			\int_K\|\Hess u\|_{\HS}^2\,d\mu
			\ge
			\frac{c}{n^2\ln n}.
		\end{equation}
	\end{lemma}
	
	\begin{proof}
		Let
		\[
		E=\int_K\|\Hess u\|_{\HS}^2\,d\mu.
		\]
		By Lemma \ref{thm:EK}, there is an absolute constant $C>0$ such that
		\[
		\int_K|\nabla V|\,d\mu\le Cn.
		\]
		Set
		\[
		A=\{x\in K:\ |\nabla V(x)|\le 4Cn\}.
		\]
		By Markov's inequality,
		$$1-\mu(A)=\mu(\{x\in K:\ |\nabla V(x)|> 4Cn\}) \le  \frac{1}{4Cn}\int_K|\nabla V|\,d\mu\le \frac{1}{4}.$$
		So we have $\mu(A)\ge \frac{3}{4}$.
		
		Since both $u$ and $\mu$ are even, each $\partial_i u$ is odd and $\int_K\partial_i u\,d\mu=0$ for $i=1,2,\ldots,n$. Applying the Poincar\'e inequality \eqref{eq:poincare} to $\partial_i u$ and summing over $i=1,\ldots,n$, we obtain
		\begin{equation}\label{eq:grad-hess}
			\int_K|\nabla u|^2\,d\mu
			=\sum_{i=1}^n\int_K(\partial_i u)^2\,d\mu
			\le
			\vartheta_\mu^2
			\sum_{i=1}^n\int_K|\nabla\partial_i u|^2\,d\mu
			=
			\vartheta_\mu^2E.
		\end{equation}
		
		Integrating the equation $L_Vu=1$ over $A$ gives
		\begin{equation*}
			\mu(A)=\int_A\Delta u\,d\mu-\int_A\langle\nabla V,\nabla u\rangle\,d\mu.
		\end{equation*}
		By Cauchy--Schwarz inequality and that $(\Delta u)^2\le n\|\Hess u\|_{\HS}^2$, we have
		\begin{equation*}
			\begin{aligned}
				\big|\int_A\Delta u\,d\mu\big|
				\le
				\mu(A)^{\frac{1}{2}}\Big(\int_A(\Delta u)^2\,d\mu\Big)^{\frac{1}{2}}\le \mu(A)^{\frac{1}{2}}\Big(\int_K(\Delta u)^2\,d\mu\Big)^{\frac{1}{2}}
				\le 
				\sqrt n\,\mu(A)^{\frac{1}{2}}E^{\frac{1}{2}}.
			\end{aligned}
		\end{equation*}
		By the definition of $A$, we have $|\nabla V|\le4Cn$ on $A$. This, together with Cauchy--Schwarz inequality and \eqref{eq:grad-hess} yields that 
		\begin{equation*}
			\begin{aligned}
				\big|\int_A\langle\nabla V,\nabla u\rangle\,d\mu\big|
				&\le
				4Cn\int_A|\nabla u|\,d\mu\le 4Cn\mu(A)^{\frac{1}{2}}\Big(\int_A|\nabla u|^2\,d\mu\Big)^{\frac{1}{2}}\\
				&\le 4Cn\mu(A)^{\frac{1}{2}}\Big(\int_K|\nabla u|^2\,d\mu\Big)^{\frac{1}{2}}
				\le
				4Cn\,\vartheta_\mu\,\mu(A)^{\frac{1}{2}}E^{\frac{1}{2}}.
			\end{aligned}
		\end{equation*}
		Combining the preceding estimates, we have
		\[
		\mu(A)\le \big|\int_A\Delta u\,d\mu\big|+\big|\int_A\langle\nabla V,\nabla u\rangle\,d\mu\big|\le 
		\mu(A)^{\frac{1}{2}}\big(\sqrt n+4Cn\vartheta_\mu\big)E^{\frac{1}{2}}.
		\]
		Since $\mu(A)\ge\frac{3}{4}$,
		\[
		E\ge
		\frac{3}{4(\sqrt n+4Cn\vartheta_\mu)^2}
		\ge
		\frac{c}{n+n^2\vartheta_\mu^2}.
		\]
		This proves \eqref{eq:energy-theta}. The bound \eqref{eq:energy-log} follows from Theorem \ref{thm:KLS}.
	\end{proof}

	\begin{proof}[Proof of Theorem \ref{main}]
		We divide the proof into three steps. We first prove that \begin{equation}\label{eq:target-local}
			p(\mu,K)\ge \frac{c}{n^2\ln n}
		\end{equation}
		for every $K\in\mathcal K^n_{+,s}$ when $\mu$ is an even log-concave measure on $\mathbb{R}^n$ with positive $C^2$ density. This is done in Step 1 and Step 2. Then we prove the general case in Step 3 by approximation.
		
		\textbf{Step 1.} Isotropic normalization of the restricted measure.
		
		Fix $K\in\mathcal K^n_{+,s}$ and consider the normalized restriction $\mu_K$ defined by
		\begin{equation}\label{eq:restricted-measure}
			d\mu_K(x)=\frac{{\bf 1}_K(x)e^{-V(x)}}{\mu(K)}\,dx.
		\end{equation}
		This is an even log-concave probability measure. Since $K$ has nonempty interior and $e^{-V}$ is positive, $\Cov(\mu_K)$ is positive definite. Let $T=\Cov(\mu_K)^{-\frac{1}{2}}$ and set
		\[
		\nu=T_\#\mu_K,\quad \widetilde K=TK.
		\]
		Then $\nu$ is even, isotropic, log-concave, and supported on $\widetilde K$ and $\widetilde K\in\mathcal K^n_{+,s}$.
		Let 
		\[
		\widetilde{V}(y)=V\left(T^{-1} y\right)+\ln |\operatorname{det} T| \quad\text{and}\quad \widetilde\mu=T_\#\mu.
		\]
		Then $d \widetilde{\mu}=e^{-\widetilde{V}} d x$. By Lemma \ref{lem:affine}, we have
		\begin{equation}\label{eq:affine-step}
			p(\mu,K)=p(\widetilde\mu,\widetilde K).
		\end{equation}
		By the definition of $\nu$, it is the normalized restriction of $\widetilde\mu$ to $\widetilde K$, i.e.,
		\begin{equation}\label{eq:nu-restriction}
			\nu=\frac{{\bf 1}_{\widetilde K}\widetilde\mu}{\widetilde\mu(\widetilde K)}.
		\end{equation}
		
		\textbf{Step 2.} The local variation estimate.
		
		Set
		\begin{equation}\label{eq:potentials-same}
			W=\widetilde V+\ln\widetilde\mu(\widetilde K)
			\quad\hbox{in }\widetilde K
		\end{equation}
		and $W=+\infty$ otherwise.
		From \eqref{eq:nu-restriction}, we have  $d\nu=e^{-W}dx$ in $\widetilde K$.
		Let $u\in C^2(\operatorname{int} \widetilde K)$ be an arbitrary even function satisfying
		\[
		L_{\widetilde V}u=1
		\quad\hbox{in }\widetilde K.
		\]
		By \eqref{eq:potentials-same}, this is the same equation as $L_Wu=1$ in $\widetilde K$. Recall that $\nu$ is even, isotropic, log-concave and supported on $\widetilde K$. Lemma \ref{lem:energy} gives
		\[
		\int_{\widetilde K}\|\Hess u\|_{\HS}^2\,d\nu
		\ge
		\frac{c}{n^2\ln n}.
		\]
		Since $\widetilde V$ is convex and $C^2$, $\Hess\widetilde V$ is positive semidefinite, and therefore $\langle\Hess\widetilde V\nabla u,\nabla u\rangle\ge 0$. Hence
		\[
		\begin{aligned}
			&\frac{1}{\widetilde\mu(\widetilde K)}
			\int_{\widetilde K}
			\big(
			\|\Hess u\|_{\HS}^2+
			\langle\Hess\widetilde V\nabla u,\nabla u\rangle
			\big)d\widetilde\mu\\
			\ge&
			\frac{1}{\widetilde\mu(\widetilde K)}
			\int_{\widetilde K}
			\|\Hess u\|_{\HS}^2
			d\widetilde\mu
			=
			\int_{\widetilde K}\|\Hess u\|_{\HS}^2\,d\nu
			\ge
			\frac{c}{n^2\ln n}.
		\end{aligned}
		\]
		Then Theorem \ref{thm:KM} implies that
		\[
		p(\widetilde\mu,\widetilde K)
		\ge
		\frac{c}{n^2\ln n}.
		\]
		This, together with \eqref{eq:affine-step}, gives \eqref{eq:target-local}.
		
		By Lemma \ref{lem:lTog},  \eqref{eq:target-local} gives \eqref{eq:mainBM} for even log-concave measures on $\mathbb{R}^n$ with positive $C^2$ density and for symmetric convex bodies.
		
		\textbf{Step 3.} Approximation and globalization.
		
		We first remove the regularity assumption on the density. Let $\gamma_\eps$ be the centered Gaussian measure with covariance $\eps^2\Id$ and density $\varphi_\eps$,  and $\mu_\eps=\mu*\gamma_\eps$. Since $d\mu=e^{-V}dx$, $\mu_\eps$ has density
		\[
		f_\eps(x)=\int_{\R^n}e^{-V(y)}\varphi_\eps(x-y)\,dy .
		\]
		Theorem \ref{thm:prekopa} shows that $f_\eps$ is log-concave. Since convolution with a Gaussian gives a $C^\infty$ positive density, and both  $\mu$ and $\gamma_\eps$ are even, $\mu_\eps$ is an even log-concave probability measure with a $C^\infty$ density. 
		Therefore, for all symmetric convex bodies $K,L$ in $\mathbb{R}^n$ and all $\lambda\in[0,1]$
		\begin{equation}\label{eq:mue}
			\mu_\eps(\lambda K+(1-\lambda)L)^{p_n}
			\ge
			\lambda\mu_\eps(K)^{p_n}+(1-\lambda)\mu_\eps(L)^{p_n}.
		\end{equation}
		
		Since $\gamma_\eps$ converges weakly to the Dirac mass $\delta_0$, $\mu_\eps$ converges weakly to $\mu$.
		So for all symmetric convex bodies $K,L$ in $\mathbb{R}^n$ and all $\lambda\in[0,1]$, 
		\[
		\mu_\eps(\lambda K+(1-\lambda)L)\to\mu(\lambda K+(1-\lambda)L)\quad \text{as}\ \eps\to0.
		\]
		Then letting $\eps\to0$, \eqref{eq:mue} gives \eqref{eq:mainBM}.
		
		It remains to pass from convex bodies to convex sets. For $R>0$ and $\delta>0$, set
		\[
		K_{R,\delta}=(K\cap RB_2^n)+\delta B_2^n,\qquad
		L_{R,\delta}=(L\cap RB_2^n)+\delta B_2^n .
		\]
		These are symmetric convex bodies, so the already proved inequality applies to $K_{R,\delta}$ and $L_{R,\delta}$. Letting $\delta \downarrow 0$ gives the target inequality for $K\cap RB_2^n$ and $L\cap RB_2^n$. Finally, as $R\to\infty$, by the monotone convergence theorem, for all symmetric convex sets $K,L$ in $\mathbb{R}^n$ and all $\lambda\in[0,1]$,
		\[
		\mu(\lambda(K\cap RB_2^n)+(1-\lambda)(L\cap RB_2^n))\to \mu(\lambda K+(1-\lambda) L).
		\]
		Then letting $R\to\infty$, we obtain \eqref{eq:mainBM}.
	\end{proof}
	
	\begin{remark}
		The proof above yields a uniform lower bound of order \(\frac{1}{n^2\ln n}\). The bound comes from the estimates
		\[
		\int_{\operatorname{supp} \mu}|\nabla V|\,d\mu\leq Cn
		\quad\text{and}\quad
		\vartheta_\mu\leq C\sqrt{\ln n},
		\]
		where $\mu$ is an even, isotropic, log-concave probability measure.
		
		By Lemma~\ref{optimal}, the order \(n\) in the first estimate is optimal. Consequently, using only this estimate and Markov's inequality, one	can guarantee a set of fixed positive measure on which	\(|\nabla V|\leq R\) only with \(R\) of order \(n\). Since the denominator in	Lemma~\ref{lem:energy} is
		\[
		\bigl(\sqrt n+R\vartheta_\mu\bigr)^2,
		\]
		even a dimension-free bound for \(\vartheta_\mu\), as predicted by the KLS	conjecture, would improve lower bound only to order \(\frac{1}{n^2}\).
		Therefore, any improvement beyond the order \(\frac{1}{n^2}\), and in particular the conjectural order \(\frac{1}{n}\), requires new ideas.
	\end{remark}

	
	\vskip3pt {\bf Data Availability}: Not applicable.


\vskip 30pt
	
	\begin{thebibliography}{99}
		\vskip 10pt
		
		\bibitem{Aishwarya-Li-2025}
		G. Aishwarya and D. Li, Entropic and functional forms of the dimensional Brunn--Minkowski inequality in Gauss space, Math. Ann. 393 (2025), 3025--3042.
		
		\bibitem{Aishwarya-Rotem-2026}
		G. Aishwarya and L. Rotem, New Brunn--Minkowski and functional inequalities via convexity of entropy, Adv. Math. 490 (2026), Paper No. 110841, 33 pp.
		
		\bibitem{Barthe}
		F. Barthe, The Brunn--Minkowski theorem and related geometric and functional inequalities, in \emph{International Congress of Mathematicians. Vol. II}, Eur. Math. Soc., Z\"urich, 2006, pp. 1529--1546.
		
		\bibitem{Bobkov1}
		S. Bobkov and M. Ledoux, From Brunn--Minkowski to sharp Sobolev inequalities, Ann. Mat. Pura Appl. 187 (2008), 369--384.
		
		\bibitem{Bobkov2}
		S. Bobkov and M. Ledoux, From Brunn--Minkowski to Brascamp--Lieb and to logarithmic Sobolev inequalities, Geom. Funct. Anal. 10 (2000), 1028--1052.
		
		\bibitem{BLYZ12}
		K. B\"or\"oczky, E. Lutwak, D. Yang and G. Zhang, The log-Brunn--Minkowski inequality, Adv. Math. 231 (2012), 1974--1997.
		
		\bibitem{Buser1982ASENP}
		P. Buser, A note on the isoperimetric constant, Ann. Sci. \'Ecole Norm. Sup. (4) 15 (1982), 213--230.
		
		\bibitem{CLM17}
		A. Colesanti, G. Livshyts and A. Marsiglietti, On the stability of Brunn--Minkowski type inequalities, J. Funct. Anal. 273 (2017), 1120--1139.
		
		\bibitem{Colesanti-2008}
		A. Colesanti, From the Brunn--Minkowski inequality to a class of Poincar\'e-type inequalities, Commun. Contemp. Math. 10 (2008), 765--772.
		
		\bibitem{Cordero-Erausquin-Eskenazis-preprint}
		D. Cordero-Erausquin and A. Eskenazis, Concavity principles for weighted marginals, arXiv:2506.16941 (2025).
		
		\bibitem{Cordero-Rotem-2023}
		D. Cordero-Erausquin and L. Rotem, Improved log-concavity for rotationally invariant measures of symmetric convex bodies, Ann. Probab. 51 (2023), 987--1003.
		
		\bibitem{Eldan-Klartag-2008}
		R. Eldan and B. Klartag, Pointwise estimates for marginals of convex bodies, J. Funct. Anal. 254 (2008), 2275--2293.
		
		\bibitem{AAN}
		A. Eskenazis, A. Giannopoulos and N. Tziotziou, Functional perimeter and the dimensional Brunn--Minkowski inequality for log-concave measures, arXiv:2605.02747 (2026).
		
		\bibitem{Eskenazis-Moschidis-2021}
		A. Eskenazis and G. Moschidis, The dimensional Brunn--Minkowski inequality in Gauss space, J. Funct. Anal. 280 (2021), Paper No. 108914, 19 pp.
		
		\bibitem{Gardner-2002}
		R. Gardner, The Brunn--Minkowski inequality, Bull. Amer. Math. Soc. (N.S.) 39 (2002), 355--405.
		
		\bibitem{Gardner-2006}
		R. Gardner, \emph{Geometric Tomography}, second ed., Cambridge University Press, Cambridge, 2006.
		
		\bibitem{Gardner-Zvavitch-2010}
		R. Gardner and A. Zvavitch, Gaussian Brunn--Minkowski inequalities, Trans. Amer. Math. Soc. 362 (2010), 5333--5353.
		
		\bibitem{KLS}
		R. Kannan, L. Lov\'asz and M. Simonovits, Isoperimetric problems for convex bodies and a localization lemma, Discrete Comput. Geom. 13 (1995), 541--559.
		
		\bibitem{Klartag-2023}
		B. Klartag, Logarithmic bounds for isoperimetry and slices of convex sets, Ars Inveniendi Analytica (2023), Paper No. 4, 17 pp.
		
		\bibitem{KL2025BAMS}
		B. Klartag and J. Lehec, Isoperimetric inequalities in high-dimensional convex sets, Bull. Amer. Math. Soc. (N.S.) 62 (2025), 575--642.
		
		\bibitem{Kolesnikov-Livshyts-2021}
		A. Kolesnikov and G. Livshyts, On the Gardner--Zvavitch conjecture: symmetry in inequalities of Brunn--Minkowski type, Adv. Math. 384 (2021), Paper No. 107689, 23 pp.
		
		\bibitem{milman2}
		A. Kolesnikov and E. Milman, Brascamp--Lieb-type inequalities on weighted Riemannian manifolds with boundary, J. Geom. Anal. 27 (2017), 1680--1702.
		
		\bibitem{Kolesnikov-Milman-2018}
		A. Kolesnikov and E. Milman, Poincar\'e and Brunn--Minkowski inequalities on the boundary of weighted Riemannian manifolds, Amer. J. Math. 140 (2018), 1147--1185.
		
		\bibitem{Kolesnikov-Milman-2022}
		A. Kolesnikov and E. Milman, Local $L^p$-Brunn--Minkowski inequalities for $p<1$, Mem. Amer. Math. Soc. 277 (2022), no. 1360, v+78 pp.
		
		\bibitem{Ledoux1994PAMS}
		M. Ledoux, A simple analytic proof of an inequality by P. Buser, Proc. Amer. Math. Soc. 121 (1994), 951--959.
		
		\bibitem{Livshyts-2023}
		G. Livshyts, A universal bound in the dimensional Brunn--Minkowski inequality for log-concave measures, Trans. Amer. Math. Soc. 376 (2023), 6663--6680.
		
		\bibitem{LMNZ17}
		G. Livshyts, A. Marsiglietti, P. Nayar and A. Zvavitch, On the Brunn--Minkowski inequality for general measures with applications to new isoperimetric-type inequalities, Trans. Amer. Math. Soc. 369 (2017), 8725--8742.
		
		\bibitem{MMRR-preprint}
		A. Malliaris, J. Melbourne, C. Roberto and M. Roysdon, Functional liftings of restricted geometric inequalities, arXiv:2508.15247 (2025).
		
		\bibitem{milman2009invent}
		E. Milman, On the role of convexity in isoperimetry, spectral-gap and concentration, Invent. Math. 177 (2009), 1--43.
		
		\bibitem{Prekopa}
		A. Pr\'ekopa, Logarithmic concave measures with application to stochastic programming, Acta Sci. Math. (Szeged) 32 (1971), 301--316.
		
		\bibitem{Saroglou2015}
		C. Saroglou, Remarks on the conjectured log-Brunn--Minkowski inequality, Geom. Dedicata 177 (2015), 353--365.
		
		\bibitem{Saroglou2016}
		C. Saroglou, More on logarithmic sums of convex bodies, Mathematika 62 (2016), 818--841.
		
		\bibitem{Schneider-book}
		R. Schneider, \emph{Convex Bodies: The Brunn--Minkowski Theory}, second expanded ed., Cambridge University Press, Cambridge, 2014.
		
	\end{thebibliography}
\end{document}